\documentclass[preprint,times,letter]{elsarticle}
\usepackage{amsmath,amsfonts,amssymb,amsthm}
\usepackage{waveletnotation}

\newtheorem{theorem}{Theorem}

\newtheorem{corollary}{Corollary}
\begin{document}

\title{Explicit Bernstein type inequalities for wavelet coefficients in $L_p(\R^n)$}

\author[rvt]{Susanna Spektor \corref{cor2}}
\ead{sspektor@math.ualberta.ca}
%\address[rvt1]{Department of Mathematical and Statistical Sciences,
%University of Alberta, Edmonton, Alberta, Canada T6G 2G1. }

\author[rvt]{Xiaosheng Zhuang\corref{cor1}}
\ead{xzhuang@math.ualberta.ca}
\address[rvt]{Department of Mathematical and Statistical Sciences,
University of Alberta, Edmonton, Alberta, Canada T6G 2G1. }

%\cortext[cor2]{Corresponding author, {\tt
%http://www.ualberta.ca/$\sim$bhan}}
%
% \cortext[cor1]{Corresponding
%author, {\tt http://www.ualberta.ca/$\sim$xzhuang}}

%
%\fntext[fn1]{Research was supported in part by NSERC Canada under
%Grant RGP 228051.}

\makeatletter \@addtoreset{equation}{section} \makeatother

\begin{abstract}In this paper, we investigate the wavelet coefficients  for  function spaces $\mathcal{A}_k^p:=\{f:
\|(\iu \omega)^k\wh{f}(\omega)\|_p\le 1\}, k\in\N, p\in(1,\infty)$
using an important quantity $C_{k,p}(\psi)$. In particular,
Bernstein type inequalities associated with wavelets are
established. We obtained a sharp inequality of Bernstein type for
splines, which induces a lower bound for the quantity
$C_{k,p}(\psi)$ with $\psi$ being the semiorthogonal spline
wavelets. We also study  the asymptotic behavior of wavelet
coefficients for both the family of Daubechies orthonormal wavelets
and the family of semiorthogonal spline wavelets. Comparison of
these two families is done by using the quantity $C_{k,p}(\psi)$.
\end{abstract}


%\subjclass[2000]{42C40, 41A05, 42C15, 65T60}

\maketitle \pagenumbering{arabic}

%==============INTRODUCTION==================================================%

\section{Introduction and Motivations}

We say that $\varphi:\R\rightarrow \C$ is a \emph{$2$-refinable
function } if
\begin{equation}\label{refeq}
\varphi=2\sum_{\nu\in\Z} a(\nu)\varphi(2 \cdot-\nu),
\end{equation}
where $a:\Z\mapsto \C$ is a finitely supported
 sequence of complex numbers on $\Z$, called the \emph{low-pass
filter (or mask)} for $\varphi$. In frequency domain, the refinement
equation in \eqref{refeq} can be rewritten as
\begin{equation}\label{refeq:freq}
\wh\varphi(2 \omega) = \wh{a}(\omega)\wh\varphi(\omega), \quad
\omega\in\R,
\end{equation}
where $\wh{a}$ is the \emph{Fourier series} of $a$ given by
\begin{equation}\label{Fseq}
\wh{a}(\omega):=\sum_{\nu\in\Z}a(\nu)e^{-\iu \nu\omega}, \quad
\omega\in\R.
\end{equation}
The Fourier transform $\wh{f}$ of $f\in L_1(\R)$ is defined to be
$\wh{f}(\omega)=\frac{1}{\sqrt{2\pi}}\int_{\R} f(x) e^{-\iu x\omega}
dx$ and can be extended to square integrable functions and tempered
distributions.
%
%We say that a compactly supported $2$-refinable function vector
%$\varphi$ in $L_2(\R)$ is \emph{orthogonal} if
%%
%\begin{equation}\label{def:orthPhi}
%\la \varphi,\varphi(\cdot-k)\ra=\delta(k), \quad k\in\Z,
%\end{equation}
%
%where $\delta$ is the \emph{Dirac sequence} such that $\delta(0)=1$
%and $\delta(k) = 0 $ for all $k\neq 0$.

Usually, a wavelet system is generated by some wavelet function
$\psi$ from a $2$-refinable function vector $\varphi$ as follows:
\begin{equation}\label{waveletGenerators}
\psi=2\sum_{\nu\in\Z}b(\nu)\varphi(2\cdot-\nu)\quad\mbox{or}\quad\wh{\psi}(2\omega)
=\wh b(\omega)\wh\varphi(\omega),
\end{equation}
where  $b:\Z\mapsto\C$ is a finitely supported sequence of complex
numbers on $\Z$, called the \emph{high-pass filter (or mask)} for
$\psi$.

%\[
%f\sim \sum_{\nu\in\Z}\la f,
%\varphi_{0,\nu}\ra\wt{\varphi}_{0,\nu}+\sum_{j\ge0}\sum_{\nu\in\Z}
%\la f, \psi_{j,\nu}\ra\wt{\psi}_{j,\nu}
%\]
Many wavelet applications, for example, image/signal compression,
are based on investigation of the wavelet coefficients $\la
f,\varphi_{j,\nu}\ra$ and  $\la f, \psi_{j,\nu}\ra$ for
$j,\nu\in\Z$, where $\la f,g\ra:=\int_\R f(x)\overline{g(x)}dx$ and
$\varphi_{j,\nu}:=2^{j/2}\varphi(2^j\cdot-\nu),\psi_{j,\nu}:=2^{j/2}\psi(2^j\cdot-\nu)$.
 The magnitude of the wavelet coefficients depends on both the
smoothness of the function $f$ and the wavelet $\psi$. In this
paper, we investigate the quantity
\begin{equation}\label{def:Ckp}
C_{k,p}(\psi)=\sup_{f\in \mathcal{A}_k^{p'}}\frac{|\la
f,\psi\ra|}{\|\wh\psi\|_p},
\end{equation}
where $1< p,p'< \infty$, $1/p'+1/p=1$, $k\in\N$, and
$\mathcal{A}_k^{p'}:=\{f\in L_{p'}(\R): \|(\iu
\omega)^k\wh{f}(\omega)\|_{p'}\le 1\}$.
 The classical Bernstein inequality states that for any
$\alpha\in\N^n$, one have $\|\partial^\alpha f\|_p\le
R^{|\alpha|}\|f\|_p$, where $f\in L_p(\R^n)$ in an arbitrary
function whose Fourier transform $\wh f(\omega)$ is supported in the
ball $|\omega|\le R$. The quantity $C_{k,p}(\psi)$ in
\eqref{def:Ckp} is the best possible constant in the following
Bernstein type inequality
\begin{equation}\label{BernsteinTypeIneq}
|\la f, \psi_{j,\nu}\ra|\le
C_{k,p}(\psi)2^{-j(k+1/p-1/2)}\|\psi\|_p\|(\iu \omega)^k \wh
f(\omega)\|_{p'}
\end{equation}
Such type of inequalities plays an important role in wavelet
algorithms for the numerical solution of integral equations (cf.
\cite{Beylkin.Coifman.Rokhlin:1991,Unser:1997}) where wavelet
coefficients arise by applying an integral operator to a wavelet and
bound of the type \eqref{BernsteinTypeIneq} gives  priori
information on the size of the wavelet coefficients.

Note that
\begin{equation}\label{def:Ckp2}
C_{k,p}(\psi)=\sup_{f\in \mathcal{A}_k^{p'}}\frac{|\la
f,\psi\ra|}{\|\wh\psi\|_p}=\sup_{f\in \mathcal{A}_k^{p'}}\frac{|\la
\wh{f},\wh{\psi}\ra|}{\|\wh\psi\|_p}=\frac{\|\wh{_k\psi}\|_p}{\|\wh{\psi}\|_p},
\end{equation}
where for a function $f\in L_1(\R)$, $_kf$ is defined to be the
function such that
\begin{equation}\label{def:kf}
\wh{_kf}(\omega)=(\iu\omega)^{-k}\wh f(\omega).
\end{equation}
For $\psi$ that is
compactly supported, it is easily shown that the quantity
$C_{k,p}(\psi)<\infty$ is euivalent to
\begin{equation}\label{def:vanishingMoment}
\int_\R \psi(x)x^\nu dx = 0 \quad\mbox{or}\quad
\frac{d^\nu}{dx^\nu}\wh\psi(0):={\wh\psi}^{(\nu)}(0)=0
\end{equation}
for $\nu =0,\ldots,m-1$. That is, $\psi$ has \emph{$m$ vanishing
moments}. Consequently, for a wavelet $\psi$ with $m$ vanishing
moments, we can investigate the magnitude of the wavelet
coefficients in the function spaces
$\mathcal{A}_1^{p'},\ldots,\mathcal{A}_m^{p'}$ for $1<p'<\infty$
using the quantity $C_{k,p}(\psi)$.

 A fundamental question in wavelet application is which wavelet one
should be choose for a specific purpose. In \cite{Keinert:1994},
Keinert used a constant $G_M$ in the following approximation for
comparison of wavelets.
\begin{equation}\label{def:GM}
\int_\R f(x)\psi_{j,\nu}(x)dx\approx
2^{-(j+1)(M+1/2)}\frac{G_M}{M!}f^{(M)}(2^{-j}\nu),
\end{equation}
where $f$ is sufficient smooth, $\psi$ has exactly $M$ vanishing
moments, and $G_M$ depends only on $\psi$. Keinert presented
numerical values of $G_M$ for some commonly used wavelets and
provided constructions for wavelets with short support an  minimal
$G_M$, which lead to better compression in practical calculation. By
considering the quantity $C_{k,p}(\psi)$, the ``$\approx$'' in
\eqref{def:GM} can be replaced by precise inequality. In
\cite{Ehrich:2000}, Ehrich investigate the quantity $C_{k,p}(\psi)$
for $p=2$ and for  two important families of wavelets. Precise
asymptotic relations of quantities $C_{k,2}(\psi)$ are established
in \cite{Ehrich:2000} showing that the quantity for the family of
semiorthogonal spline wavelts are generally smaller than that for
the family of Daubechies orthonormal wavelets.

%
% According to various requirements of problems in applications,
%different desired properties of a wavelet system are needed. Among
%all properties of a wavelet system, high order of vanishing moments,
%orthogonality, and rational coefficient masks are highly desirable
%properties in wavelet analysis and  its applications. High order of
%vanishing moments is crucial for the sparsity representation of a
%wavelet system, which plays an important role in image denoising and
%compression. Orthogonality results in simple rules for guaranteeing
%the perfect reconstruction property. Rational coefficients  means
%exact decomposition and reconstruction in the implementation of the
%fast wavelet algorithm. However, these properties usually mutually
%conflict to each other.

 In this paper, we shall investigate the quantity $C_{k,p}(\psi)$ mainly  for
 the family of
Daubechies orthonormal wavelets (see \cite{Daub:book}) and the
family of semiorthogonal spline wavelets (see
\cite{Chui.Wang:1992}). Let $m$ be a positive integer. The
Daubechies orthonormal wavelet $\psi^D_m$ of order $m$ with mask
$b^D_m$  and its $2$-refinable function $\varphi^D_m$ with mask
$a^D_m$ are determined by
\begin{equation}\label{def:DaubWave}
\begin{aligned}
|\wh{a^D_m}(\omega)|^2&=\cos^{2m}(\omega/2)\sum_{\nu=0}^{m-1}{m-1+\nu
\choose \nu}\sin^{2\nu}(\omega/2),\\
\wh{b^D_m}(\omega)&=e^{-\iu\omega/2}\overline{\wh{a^D_m}(\omega/2+\pi)},
\end{aligned}
\end{equation}
while the simiorthogonal spline wavelet $\psi^S_m$ of order $m$ is
given by
\begin{equation}\label{def:SplineWave}
\begin{aligned}
\psi^S_m(x)=\sum_{\nu=0}^{2m-2}\frac{(-1)^\nu}{2^{m-1}}N_{2m}(\nu+1)N_{2m}^{(m)}(2x-\nu),
\quad x\in\R,
\end{aligned}
\end{equation}
where $N_m$ is the B-spline of order $m$. That is,
\begin{equation}\label{def:Bspline}
N_m(x) = \frac{1}{(m-1)!}\sum_{\nu=0}^{m}(-1)^\nu{m\choose
\nu}(x-\nu)^{m-1}_+,
\end{equation}
or equivalently,  $\wh{N_m}(\omega)=\frac{1}{\sqrt{2\pi}}(e^{-\iu
\omega/2}\frac{\sin(\omega/2)}{\omega/2})^m$. Here for $k\ge 1$,
\[
(y)_+^k=
\begin{cases}
y^k & y>0,\\
0,  &y\le0,
\end{cases}
\quad\mbox{and}\quad (y)_+^0=
\begin{cases}
1 & y>0,\\
\frac12,  &y=0,\\
0,& y<0.
\end{cases}
\]
Note that $\psi^S_m$ is generated from the $2$-refinable function
$\varphi^S_m:=N_m$  via \eqref{waveletGenerators} by some mask
$b^S_m$ (cf. \cite{Chui.Wang:1992}).

These two families are widely used in many applications. For
example, see
\cite{Beylkin.Coifman.Rokhlin:1991,Petersdorff.Schwab:1996,Rathsfeld:1995,Unser:1997}
for their applications on numerical solution of PDE and signal/image
processing.  Both of the Daubechies orthonormal wavelet $\psi^D_m$
and the semiorthogonal spline wavelet $\psi^S_m$ have vanishing
moments of order $m$ and support length $2m-1$. The Daubechies
orthonormal wavelet $\psi^D_m$ generates an orthonormal basis
$\{2^{j/2}\psi^D_m(2^j\cdot-\nu): j,\nu\in\Z\}$ for $L_2(\R)$ (see
\cite{Daub:book}). However, the wavelet function $\psi^D_m$ is
implicitly defined and the coefficients in the mask for $\psi^D_m$
are not rational numbers. In fact,
\begin{equation}\label{def:DaubPsi}
\wh{\psi^D_m}(\omega)=\frac{1}{\sqrt{2\pi}}\wh{a^D_m}(\omega/2+\pi)\prod_{\ell=1}^\infty\wh{a^D_m}(2^{-\ell}\omega)
\end{equation}
and $a^D_m$ is obtained from \eqref{def:DaubWave} via Riesz lemma.
The semiorthogonal spline wavelets generated by $\psi^S_m$ are not
orthogonal in the same level $j$. Yet they are orthogonal on
different levels. And more importantly, the semiorthogonal spline
wavelet $\psi^S_m$ is explicit defined and the coefficients for its
mask are indeed rational numbers, which is a very much desirable
property in the implementation of fast wavelet algorithms. We shall
 see that these two families significantly differ with respect
to the magnitude of their wavelet coefficients in terms of
$C_{k,p}(\psi^D_m)$ and $C_{k,p}(\psi^S_m)$.

The structure of this paper is  as follows. In Section~2, for $k,
m\in\N$ fixed and $p\in(1,\infty)$, we shall investigate the
quantity $C_{k,p}(\psi^S_m)$ in the Bernstein type inequality in
\eqref{BernsteinTypeIneq} for the familiy of semiorthogonal spline
wavelets. In Section~3, we shall establish results on the asymptotic
behaviors ($m\rightarrow\infty$) of the quantities
$C_{k,p}(\varphi)$ and $C_{k,p}(\psi)$ for both the refinable
function $\phi$ and wavelet function $\psi$ and for both the two
families of wavelets. Finally, we shall generalize our results to
high-dimensional wavelets in Section~4.

%\subsection{Notations and Preliminaries.}
%
%We consider continuous wavelets $\psi_m$ of compact support $[0, L]$ which associated with relation,
%\[
%\psi_m(x)=2 \sum_{\nu=0}^{2L-L^*}g_\nu \varphi_m(2x-\nu),
%\]
%where $\varphi_m$ is the corresponding scaling function of compact support $[0, L^*]$ satisfying
%\[
%\varphi_m=2 \sum_{\nu=0}^{L^*}h_\nu \varphi_m(2x-\nu).
%\]
%The $k$-th derivative of wavelet $\psi_m(x)$ is $_k\psi_m(x)=2^{1-k}\sum_{\nu=0}^{2L-L^*}g_{\nu} \ _k\varphi(2x-\nu)$.
%
%And $B$-spline wavelet $\psi_m^S(x)$ is defined [3] by
%\[
%\displaystyle{\psi_m^S(x)=\sum_{\nu=0}^{2m-2}\frac{(-1)^\nu}{2^{m-1}}N_{2m}(\nu+1)N_{2m}^{(m)}(2x-\nu), \textit{ $x \in \R$ }}.
%\]

%=============================SPLINE WAVELETS WITH FIXED MOMENT=========================%

\section{Bernstein Type Inequalities for Splines }
In this section, we shall first establish a result of the Bernstein
type inequality for splines and then present an upper bound for the
quantity $C_{k,p}(\psi^S_m)$. Throughout this paper, $p\in \R$
always denotes a constant such that $p\in(1,\infty)$.

Before we introducing our results, we need some notation and
definitions.

A function $s(x)$ is called a spline of order $m$ of minimal defect
with nodes $lh, h>0, l \in \Z$, if
\begin{itemize}
\item[{\rm(1)}] $s(x)$ is a polynomial with real coefficients of the degree
$<m$ at each interval $(h(l-1), hl)$, $l \in \Z$;

\item[{\rm(2)}] $s(x) \in C^{m-2}(\R).$
\end{itemize}
The collection of all such splines is denoted by $S_{m,h}$. It is
well known that any spline $s \in S_{m,h}$ can be  uniquely
represented by
\begin{equation}\label{def:splines}
s(x)= \sum_{\nu \in \Z}c_{\nu} N_{m}(x-h\nu).
\end{equation}
Here $N_m$ is the B-spline of order $m$. It is well known that
\begin{equation}\label{eq:NmProperty1}
N_m'(x)=N_{m-1}(x)-N_{m-1}(x-1)\quad\mbox{for}\quad m\ge 2.
\end{equation}
and
\begin{equation}\label{eq:NmProperty2}
\sum_{k= -\infty}^\infty|\wh{N_m}(\omega+2 \pi k)|^2=
\sum_{k=-m+1}^{m-1}N_{2m}(m+k)e^{-ik \omega}.
\end{equation}

The following result  provides an exact upper bound in the Bernstein
type inequality for any spline $s\in S_{m,h}$, which gives
estimation of $k$th derivative of non-periodic spline in $L_p(\R)$
by $L_p(\R)$ norm of the spline $s$ itself (also cf.
\cite{Babenko.Spektor:2008} for a special case $p=2$).

\begin{theorem}\label{thm1:spline}
Let $k,m\in\N$, $k<m$, and $h\in\Z$. Let $p\in(1,\infty)$. Then, for
any function $s \in S_{m,h}$ such that $\wh{s}\in L_p(\R)$, the
following sharp inequality holds:
\begin{equation}\label{eq:Cksplin}
\displaystyle{\|\wh{s^{(k)}}\|_p\leq \left(\pi h\right)^k
\left({\frac{K_{2(m-k)+1}}{K_{2m+1}}}\right)^{1/2}\|\wh{s}\|_p},
\end{equation}
where
$\displaystyle{K_j=\frac{4}{\pi}\sum_{\ell=0}^{\infty}\frac{(-1)^{\ell(j+1)}}{(1+2\ell)^{j+1}}}$,
$j=0,1,2,...$ are the Favard's constants.
\end{theorem}

\begin{proof}
%Let us first show that the $k$th derivative $s^{(k)}$ of a spline
%$s$, i.e.,
%\[
% s^{(k)}(x)=\sum_{\nu \in \Z}c_\nu N_m^{(k)}(x+h\nu),
%\]
%belongs to the space $L_p(\R)$. Without lost of generality, we
%assume $k=1$ and $h=1$.
%%
%%Note, $s'(x) \in L_p(\R)$,  if $\left(\int_{\R}((s'(x))^p)dx
%%\right)^{\frac1p}< \infty$.
%We have
%\[
%\begin{aligned}
%\int_{\R}|s'(x)|^pdx &=\int_{\R}\left|\sum_{\nu \in
%\Z}c_{\nu}N'_m(x+\nu)\right|^p dx\\
% &= \int_{\R}\left|\sum_{\nu \in \Z}c_{\nu} [
% N_{m-1}(x+\nu)-N_{m-1}(x+\nu-1)]\right|^pdx\\
% & =\int_{\R}\left|\sum_{\nu \in \Z}(c_\nu-c_{\nu +1})N_{m-1}(x+\nu)\right|^p dx.
%\end{aligned}
%\]
% Since the shifts of  $N_m$  is stable in $L_p(\R)$, there exist positive constants $A$ and $B$ ($0< A \leq B< \infty $) such that
%\[
% A\sum_{\nu \in \Z}|c_{\nu}-c_{\nu+1}|^p\leq \int_{\R}\left|\sum_{\nu \in \Z}(c_\nu-c_{\nu +1})N_{m-1}(x+\nu)\right|^p dx
% \leq B\sum_{\nu \in \Z}|c_{\nu}-c_{\nu+1}|^p.
%\]
%Consequently, together with $s\in L_p(\R)$, we conclude that $s'\in
%L_p(\R)$.
 We first show that \eqref{eq:Cksplin} is true for $k=1$ and $h=1$.

Since $s'(x)=\sum_{\nu \in \Z}c_\nu N_m'(x-\nu)$, by
\eqref{eq:NmProperty1}, we have
\[
\begin{aligned}
\|\wh{s'}\|_p^p &= \int_{\R}\left|\sum_{\nu \in \Z}c_\nu e^{-i \nu
\omega}\wh{N_{m-1}}(\omega)(1-e^{-i\omega})\right|^pd\omega
\\&=\int_{0}^{2
\pi}\left|\wh{a_s}(\omega)(1-e^{-i
\omega})\right|^p\sum_{\ell\in\Z}\left|\wh{N_{m-1}}(\omega+2\pi
\ell)\right|^p d\omega
\\&
=\int_{0}^{2 \pi}\frac{\left|1-e^{-i
\omega}\right|^p\sum_{\ell\in\Z} \left|\wh{N_{m-1}}(\omega+2 \pi
\ell)\right|^p} {\sum_{\ell\in\Z}\left|\wh{N_m}(\omega+2 \pi
\ell)\right|^p}\left|\wh{a_s}(\omega)\right|^p
\sum_{\ell\in\Z}\left|\widehat{N}_{m}(\omega+2 \pi \ell)\right|^p d
\omega
\\&\leq
\max_{\omega\in[0,2\pi]}\frac{\left|1-e^{-i
\omega}\right|^p\sum_{\ell\in\Z}\left|\wh{N_{m-1}}(\omega+2 \pi
\ell)\right|^p} {\sum_{\ell\in\Z}\left|\wh{N_m}(\omega+2 \pi
\ell)\right|^p}\|\wh{s}\|_p^p.
\end{aligned}
\]
Here $\wh{a_s}(\omega)=\sum_{\nu\in\Z}c_\nu e^{-i\nu\omega}$.
 Denote $L(\omega):=\frac{\left|1-e^{-i \omega}\right|^2\sum_{\ell\in\Z}\left|\wh{N_{m-1}}(\omega+2 \pi l)\right|^2}
 {\sum_{\ell\in\Z}\left|\wh{N_m}(\omega+2 \pi \ell)\right|^2}$. Let
 us find the maximum of $L(\omega)$ on $[0,2\pi]$.

 A function of complex variables $z$, determined by
\begin{equation}
 E_{2m-1}(z)=(2m-1)! z^{m-1} \sum_{k=-m+1}^{m-1}N_{2m}(m+k)z^k, \forall z \in \R
\end{equation}
 is called Euler-Frobenious' polynomials of order $2m-1$ (or degree $2m-2$).
By  \cite[p.151]{Chui:book}, we have
\[
\sum_{\ell\in\Z}\left|\wh{N_m}(\omega+2 \pi l)\right|^2
=\frac{1}{(2m-1)!}\prod_{\ell=1}^{m-1}\frac{1-2\lambda_\ell \cos
\omega +\lambda_\ell^2}{|\lambda_\ell|},
\]
and similarly,
\[
\sum_{\ell\in\Z}\left|\wh{N_{m-1}}(\omega+2 \pi \ell)\right|^2=
\frac{1}{(2m-2)!}\prod_{\ell=1}^{m-2}\frac{1-2\beta_\ell \cos \omega
+\beta_\ell^{2}}{|\beta_\ell|},
\]
where $\lambda_\ell$ and $\beta_\ell$ are the roots of
Euler-Frobenious' polynomials $E_{2m-1}(z)$ and $E_{2m-3}(z)$
respectively. Moreover,
$-1<\lambda_{m-1}<\lambda_{m-2}<..<\lambda_1<0$ ,
$-1<\beta_{m-2}<\beta_{m-3}<..<\beta_1<0$, and
$\beta_j>\lambda_{j+1}$ for $j=1,\ldots,m-2$. Then, up to a
constant, the following is true
\[
L(\omega)=(1- \cos \omega)\frac{\prod_{\ell=1}^{m-2}(1-2\beta_\ell
\cos \omega +\beta_\ell^2)} {\prod_{\ell=1}^{m-1}(1-2\lambda_\ell
\cos \omega +\lambda_\ell^2)}=:(1-\cos\omega)A(\omega).
\]
Note that
\[
\begin{aligned}
A'(\omega)&=2\sin\omega\cdot
A(\omega)\left(\sum_{\ell=1}^{m-2}\frac{\beta_{\ell}}{1-2\beta_{\ell}\cos\omega+\beta_\ell^2}
-\sum_{\ell=1}^{m-1}\frac{\lambda_\ell}{1-2\lambda_{\ell}\cos\omega+{\lambda_\ell}^2}\right).
\end{aligned}
\]
It is easy to see that $1-2\lambda_\ell \cos \omega
+\lambda_\ell^2>0$ and $1-2\beta_\ell \cos \omega
+\beta_\ell^{2}>0$. Moreover,
\[
\begin{aligned}
\beta_{\ell}(1-2\lambda_{\ell+1}\cos\omega+\lambda_{\ell+1}^2)-
\lambda_{\ell+1}(1-2\beta_{\ell}\cos\omega+\beta_\ell^2)=(\beta_\ell-\lambda_{\ell+1})
(1-\beta_\ell\lambda_{\ell+1})>0.
\end{aligned}
\]
Consequently, $L'(\omega)=0$ has only one root $\omega=\pi$ on
$(0,2\pi)$. That is, $L(\omega)$ attends its maximum at
$\omega=\pi$. Hence,
\[
\max_{\omega\in[0,2\pi]}
L(\omega)=4\frac{\sum_{\ell\in\Z}\left|\wh{N_{m-1}}(\pi+2 \pi
\ell)\right|^2} {\sum_{\ell\in\Z}\left|\widehat{N}_{m}(\pi+2 \pi
\ell)\right|^2}.
\]
As in \cite{Chui:book},
\[
\sum_{\ell\in\Z}\left|\wh{N_m}(\pi+2 \pi \ell)\right|^{2}=
\frac{2^{2m+2} \sin^{2m+2}
\frac{\omega}{2}}{\sum_{\ell\in\Z}|\omega+2 \pi \ell|^{2m+2}}.
\]
%and,  similarly,
%\[
%\displaystyle{\sum_{l}\left|\wh{N_{m-1}}(\pi+2 \pi l)\right|^{2m-2}=
%\frac{2^{2m} \sin^{2m} \frac{\omega}{2}}{\sum_l|\omega+2 \pi
%l|^{2m}}}.
%\]
Consequently,
\[
\begin{aligned}
\max_{\omega\in[0,2\pi]} L(\omega)&=4\frac{2^{2m} \sin^{2m}
\frac{\pi}{2}}{\sum_{\ell\in\Z}|\pi+2 \pi
\ell|^{2m}}\cdot\frac{\sum_{\ell\in\Z}|\pi+2 \pi
\ell|^{2m+2}}{2^{2m+2} \sin^{2m+2} \frac{\pi}{2}}
\\&=\pi^2\frac{\sum_{\ell\in\Z}|1+2 \pi
\ell|^{2m+2}}{\sum_{\ell\in\Z}|1+2 \pi
\ell|^{2m}}=\pi^2\frac{K_{2m-1}}{K_{2m+1}},
\end{aligned}
\]
where $K_{2m-1},K_{2m+1}$ are Favard's constants (cf.
\cite[p.64-65]{Kornejchuk:1984}).
% That is
%\[
%\max_\omega L(\omega)=\frac{\pi^2}{K_{2m-1}}{K_{2m+1}}.
%\]
From above calculations, we obtain,
\[
\|\wh{s'}\|_p\leq \pi
\left({\frac{K_{2m-1}}{K_{2m+1}}}\right)^{1/2}\|\wh{s}\|_p.
\]
For integral-valued shifts $h$ of splines
$
s(x)=\sum_{\nu \in \Z}c_{\nu} N_m(x+h \nu), \textit{ $h \in \Z$},
$
One can show that
\[
\|\wh{s'}\|_p\leq (\pi h)
\left({\frac{K_{2m-1}}{K_{2m+1}}}\right)^{1/2}\|\wh{s}\|_p.
\]

Now, by induction, it is easy to show that \eqref{eq:Cksplin} holds.
%\[
%||\wh{s^{(k)}}||_p \leq \left(\pi h\right)^k
%\left({\frac{K_{2(m-k)+1}}{K_{2m+1}}}\right)^{p/2}\|\wh{s}\|_p.
%\]
%for $k,m\in\N$, $k<m$, and $h\in\Z$.

 Finally, we show that the
constant in \eqref{eq:Cksplin} is the best possible one.

 Let
$|\wh{a_s}(\omega)|^p=\frac{1}{2\pi}\Phi_j(\omega-\omega_0)$ and
$\wh{s}(\omega):=\wh{a_s}(\omega)\wh{N_m}(\omega)$, where
$\Phi_j(\omega)$ is a Feyer's kernel of order $j$ and $\omega_0=\pi$
is the point which realizes the  maximum  on the right hand side of
inequality \eqref{eq:Cksplin}. Note,  $\frac{1}{2 \pi} \int_0^{2
\pi}\Phi_j(\omega)d \omega=1$. Then,
\[
\begin{aligned}
\|\wh{s'}\|_p^p&=\int_\R \left|\sum_{\nu \in \Z}c_\nu e^{-i\nu
x}\wh{N_{m-1}}(\omega)(1-e^{-i\omega})\right|^pd\omega
\\&=\int_0^{2\pi}|\wh{a_s}(\omega)(1-e^{-i\omega})|^p\sum_{\ell\in\Z}
\left|\wh{N_m-1}(\omega+2\pi l)\right|^pd\omega
\\&
=\frac{1}{2\pi} \int_0^{2 \pi}\frac{|1-e^{-i
\omega}|^p\sum_{\ell\in\Z} \left|\wh{N_{m-1}}(\omega+2 \pi
\ell)\right|^p}{\sum_{\ell\in\Z}\left|\wh{N_m}(\omega+2 \pi
\ell)\right|^p}|\Phi_j(\omega-\omega_0)|
\sum_{\ell\in\Z}\left|\wh{N_{m}}(\omega+2 \pi \ell)\right|^p d
\omega
\\&
\rightarrow \frac{|1-e^{-i
\omega_0}|^p\sum_{\ell\in\Z}|\wh{N_{m-1}}(\omega_0+2 \pi
\ell)|^p}{\sum_{\ell\in\Z}|\wh{N_m}(\omega_0+2 \pi
\ell)|^p}\|\wh{s}\|_p^p ,\qquad j\rightarrow\infty.
\end{aligned}
\]
Consequently,
\[
\frac{\|\wh{s'}\|_p^p}{\|\wh{s}\|_p^p} \rightarrow
\max_{\omega\in[0,2\pi]} \frac{|1-e^{-i
\omega}|^p\sum_{\ell\in\Z}|\wh{N_{m-1}}(\omega+2 \pi
\ell)|^p}{\sum_{\ell\in\Z}|\wh{N_m}(\omega+2 \pi \ell)|^p}=\pi^p
\left({\frac{K_{2m-1}}{K_{2m+1}}}\right)^{p/2},
\]
which completes the proof.
\end{proof}

By Theorem~\ref{thm1:spline}, obviously, we have, $C_{k,p}(s)\ge
\left(\pi \right)^{-k}
\left({\frac{K_{2(m+k)+1}}{K_{2m-1}}}\right)^{1/2}$ for any $s\in
S_{m,h}$ such that $\wh{s}\in L_p$.  By the definition of $\psi_m^S$
in \eqref{def:SplineWave}, we have the following corollary.

\begin{corollary}\label{cor:splineWaveFixmC}
Let $k\geq 0$ be a fixed integer. Then,
$$
C_{k,p}(\psi_m^S)\ge \left(\frac{1}{2\pi}\right)^k
\left(\frac{K_{2(m+k)+1}}{K_{2m+1}}\right)^{1/2}.
$$
\end{corollary}

\begin{proof}
Let $f:=_k\psi_m^S$. Then $\wh{f^{(k)}}=\wh{\psi^S_m}$. By
\eqref{def:SplineWave},
\[
f(x)=_k\psi^S_m(x)=\sum_{\nu=0}^{2m-2}\frac{(-1)^\nu}{2^{m+k-1}}N_{2m}(\nu+1)N_{2m}^{(m-k)}(2x-\nu).
\]
Consequently, $f(\cdot/2)\in S_{m+k,1}$. In view of
Theorem~\ref{thm1:spline}, we have
\[
\frac{\|\wh{f^{(k)}}\|_p}{\|\wh{f}\|_p}\le (2\pi)^k
\left(\frac{K_{2m+1}}{K_{2(m+k)+1}}\right)^{1/2}.
\]
Now, by that
$C_{k,p}(\psi^S_m)=\frac{\|\wh{f}\|_p}{\|\wh{f^{(k)}}\|_p}$, we are
done.
\end{proof}

From Corollary~\ref{cor:splineWaveFixmC}, when $m$ is large enough,
we see that $C_{k,p}(\psi^S_m)\approx (\frac{1}{2\pi})^k$. In next
section, we shall study the exact asymptotic behavior of  these
types of  quantities  as $m\rightarrow \infty$ for both the family
of Daubechies orthonormal wavelets and the family of semiorthogonal
spline wavelets.

\section{The Asymptotic Estimation of Wavelet Coefficients} In this section, we shall study  the asymptotic behavior
of wavelet coefficients  for both Daubechies orthonormal wavelets
and semiorthogonal spline wavelets. We first study the asymptotic
behavior of the wavelet coefficients for Daubechies orthonormal
wavelets the first subsection. In the second subsection, We
investigate the asymptotic behavior of the wavelet coefficients for
semiorthogonal spline wavelets. In the last subsection, we shall
 compare the asymptotic behaviors of wavelet coefficients for
these two families based on the quantity defined in \eqref{def:Ckp}.

%
%For a function $f\in L_1(\R)$, its Fourier transform is defined to
%be
%\[
%\wh{f}(\omega)=\frac{1}{\sqrt{2\pi}}\int_\R
%f(x)e^{-ix\omega}dx,\quad  \omega\in\R.
%\]
%Let $k\in\Z$ be an integer. A function $(_k f)(\cdot)$ is defined by
%its Fourier transform as
%\begin{equation}\label{def:kf}
% \wh{_k f}(\omega) = (i\omega)^{-k}\wh f(\omega),\quad
%\omega\in\R.
%\end{equation}
%A constant $C_{k,p}(f)$ depending on $f$ and $k,p$ with $k\in\Z,
%1\le p\le \infty$ is given by
%\[
%C_{k,p}(f):=\frac{\|\wh{_kf}\|_p}{\|{\wh{f}}\|_p},
%\]
%where $\|\cdot\|_p$ is the $L_p$-norm of a function in the function
%space $L_p(\R$). Let $N_m$ denote the B-spline of order $m$:
%%
%\[
%\wh
%{N_m}(\omega)=\frac{1}{\sqrt{2\pi}}\left(e^{-i\omega/2}\frac{\sin(\omega/2)}{\omega/2}\right)^m.
%\]

\subsection{The  Wavelet Coefficients of  Daubechies Orthonormal Wavelets}
Let $H_m(t)$ be a $2\pi$-periodic trigonometric function defined by
\begin{equation}\label{mask:Daub}
H_m(t)=\sum_{\nu=0}^Lh_\nu e^{-i\nu t},\quad |H_m(t)|^2
=1-c_m\int_{0}^t\sin^{2m-1}\omega d\omega,
\end{equation}
where $c_m=\left(\int_0^\pi\sin^{2m-1}\omega
d\omega\right)^{-1}=\frac{\Gamma(m+1)}{\sqrt{\pi}\Gamma(m)}\sim\sqrt{\frac{m}{\pi}}$.
Then, $H_m=a_m^D$ is the Daubechies orthonormal mask of order $m$
(cf. \cite{Ehrich:2000}).

To compare with the semiorthogonal spline wavelets, we need the
following result for the Daubechies scaling function $\varphi^D_m$.
\begin{theorem}\label{thm:CphiD}
Let $\varphi^D_m$ be the Daubechies orthonormal  scaling function of
order $m$, i.e.,
$\wh{\varphi^D_m}(\omega)=\frac{1}{\sqrt{2\pi}}\prod_{\ell=1}^\infty
H_m(2^{-\ell}\omega)$. Then
\begin{equation}\label{eq:kPhiD}
\lim_{m\rightarrow \infty}\|\wh{_{-k}\varphi^D_m}\|_p
=\pi^k\frac{(2\pi)^{1/p-1/2}}{(1+pk)^{1/p}} ,\quad k\in\N.
\end{equation}
\end{theorem}
\begin{proof}
Let $\Phi:=\frac{1}{\sqrt{2\pi}}\chi_{[-\pi,\pi]}$. We have
\[
\|\wh{_{-k}\varphi^D_m}\|_p^p
=\int_{\R}|\omega|^{pk}|\wh{\varphi^D_m}(\omega)|^p d\omega
=\int_{\R}|\omega|^{pk}|\wh{\varphi^D_m}(\omega)-\Phi(\omega)+\Phi(\omega)|^p
d\omega
\]
Note that
\[
\int_{\R}|\omega|^{pk}|\Phi(\omega)|^pd\omega=\pi^{pk}\frac{(2\pi)^{1-p/2}}{1+pk}.
\]
We next prove that
\[
I:=\int_{\R}|\omega|^{pk}|\wh{\varphi^D_m}-\Phi(\omega)|^pd\omega\rightarrow
0,\quad \mbox{as } m\rightarrow\infty.
\]
In fact,
\[
I=\int_{|\omega|>\pi}|\omega|^{pk}|\wh{\varphi^D_m}(\omega)|^pd\omega
+\int_{|\omega|\le\pi}|\omega|^{pk}|\wh{\varphi^D_m}(\omega)-\Phi(\omega)|^pd\omega=:I_1+I_2.
\]
By the regularity of $\varphi^D_m$, i.e.,
$|\wh{\varphi^D_m}(\omega)|\le C_1|\omega|^{-C_2\log(m)}$,
obviously, $I_1\rightarrow0$ as $m\rightarrow\infty$. For $I_2$, let
$I:=[-\pi,\pi]$, $\delta>0$ be fixed, and
$I_\delta:=[-\pi+\delta,\pi-\delta]$. Then
\[
I_2=\int_{I_\delta}|\omega|^{pk}|\wh{\varphi^D_m}(\omega)-\Phi(\omega)|^pd\omega+
\int_{I\backslash
I_\delta}|\omega|^{pk}|\wh{\varphi^D_m}(\omega)-\Phi(\omega)|^pd\omega:=I_{21}+I_{22}.
\]
For $I_{22}$, we have $I_{22}\le C\delta$ for some $C$ depending
only on $p,k$, since $\wh{\varphi^D_m}$ and $\Phi$ are both bounded.
For $I_{21}$, we have
\[
I_{21}\le\int_{I_\delta}|\omega|^{pk}|\wh{\varphi^D_m}(\omega)-\frac{1}{\sqrt{2\pi}}H_m(\omega/2)|^pd\omega
+\int_{I_\delta}|\omega|^{pk}|\frac{1}{\sqrt{2\pi}}H_m(\omega/2)-\Phi(\omega)|^pd\omega\rightarrow
0
\]
as $m\rightarrow \infty$ since $\frac{1}{\sqrt{2\pi}}H_m(\omega/2)$
converges to $\Phi$ uniformly in $I_\delta$ and
\[
\begin{aligned}
&\int_{I_\delta}|\omega|^{pk}|\wh{\varphi^D_m}(\omega)-\frac{1}{\sqrt{2\pi}}H_m(\omega/2)|^pd\omega
\\&\le
\int_{I_\delta}|\omega|^{pk}\left|\frac{1}{\sqrt{2\pi}}H_m(\omega/2)\left(\prod_{\ell=1}^{\infty}H_m(2^{-l-1}\omega)-1\right)\right|^pd\omega
\\&\le
\int_{I_\delta}|\omega|^{pk}\left|\frac{1}{\sqrt{2\pi}}\left(\prod_{\ell=1}^{\infty}H_m(2^{-l-1}\omega)-1\right)\right|^pd\omega
\rightarrow 0,\quad m\rightarrow \infty.
\end{aligned}
\]
Consequently, we obtain
\[
\lim_{m\rightarrow \infty}\|\wh{_{-k}\varphi^D_m}\|_p
=\pi^k\frac{(2\pi)^{1/p-1/2}}{(1+pk)^{1/p}} ,\quad k\in\N.
\]
\end{proof}
More generally,  one can also show that for $\alpha\in\R$ such that
$1-p\alpha>0$,
\begin{equation}\label{eq:alphaCphiD}
\lim_{m\rightarrow \infty}\|\wh{_{\alpha}\varphi^D_m}\|_p
=\pi^{-\alpha}\frac{(2\pi)^{1/p-1/2}}{(1-p\alpha)^{1/p}},
\end{equation}
where for a real number $\alpha\in\R$, the function
$_\alpha\varphi^D_m$ is similarly defined as in \eqref{def:kf}.
However when $1-p\alpha\le0$, i.e., $\alpha\ge 1/p$, the constant
$\|\wh{_{\alpha}\varphi^D_m}\|_p\rightarrow\infty$ as
$m\rightarrow\infty$.

When $k$ fixed and $m\rightarrow\infty$,  Babenko and Spektor
(\cite{Babenko.Spektor:2007})show that, for the Daubechies
orthonormal wavelet function $\psi^D_m$ with $m$ vanishing moments,
one has
\begin{equation}\label{eq:kPsiD}
\lim_{m\rightarrow \infty}\|\wh{_k\psi^D_m}\|_p
=\frac{(2\pi)^{1/p-1/2}}{\pi^k}\left(\frac{1-2^{1-pk}}{pk-1}\right)^{1/p}
,\quad k\in\N.
\end{equation}

When $k=m$, we can deduce the following estimation, which in turn
gives rise to  the asymptotic behavior of the constant
$[C_{m,p}(\psi^D_m)]^{1/m}$.
\begin{theorem}\label{thm:CmpsimD}
Let $\psi_m^D$ be the Daubechies wavelet with $m$ vanishing moments,
i.e., $\wh{\psi_m^D}(\omega)=\frac{1}{\sqrt{2\pi}}
H_m(\omega/2+\pi)\prod_{\ell=1}^\infty H_m(2^{-l-1}\omega)$.
%\[
%\wh{\psi_m^D}(\omega)=\frac{1}{\sqrt{2\pi}}
%H_m(\omega/2+\pi)\prod_{\ell=1}^\infty H_m(2^{-l-1}\omega)
%\]
%with
%\[
%H_m(t)=\sum_{\nu=0}^Lh_\nu e^{-i\nu t},\quad |H_m(t)|^2
%=1-c_m\int_{0}^t\sin^{2m-1}\omega d\omega,
%\]
%where $c_m=\left(\int_0^\pi\sin^{2m-1}\omega
%d\omega\right)^{-1}=\frac{\Gamma(m+1)}{\sqrt{\pi}\Gamma(m)}\sim\sqrt{\frac{m}{\pi}}$.
 Then
\begin{equation}
\|\wh{_m\psi_m^D}\|_p =
C\cdot\frac{2^{1/p}}{\sqrt{2\pi}}\cdot\frac{2^{-m}\cdot
A(m)}{(\sqrt{mp/2})^{1/p}}\cdot(1+\mathcal{O}(m^{-1/2}),
\end{equation}
where $C$ is a positive constant independent of $m$ and
$\sqrt{\frac{c_m}{2m}}\le A(m)\le \sqrt{\frac{1}{2}}$.
\end{theorem}
\begin{proof}By definition,
\[
\|\wh{_m\psi_m^D}\|_p^p = \int_\R
|\omega|^{-mp}|\wh{\psi_m^D}(\omega)|^pd\omega
=\left\{\int_{|\omega|\le\pi}+\int_{|\omega|> \pi}\right\}
|\omega|^{-mp}|\wh{\psi_m^D}(\omega)|^pd\omega =:I_1+I_2.
\]

We first estimate $I_2$. Since $|H_m(t)\le1$,
\[
I_2\le\frac{2}{\sqrt{2\pi}^p}\int_{\pi}^\infty\omega^{-mp}d\omega
\le \frac{2}{\sqrt{2\pi}^p}\cdot
\frac{1}{mp-1}\left(\frac{1}{\pi}\right)^{mp-1},\quad mp>1.
\]

Next,  we show that $I_1\sim C\cdot
c_m^{p/2}\cdot(\sqrt{mp/2})^{-1}\cdot2^{-mp}$. By the definition of
$H_m(t)$, for $\omega\in[0,\pi]$, we have
\[
\begin{aligned}
|H_m(\frac{\omega}{4})|^2\ge
1-c_m\frac{\omega}{4}\sin^{2m-1}(\frac{\omega}{4})\ge
1-c_m\frac{\pi}{4}\sin^{2m-1}(\frac{\pi}{4})\ge
1-\frac{\Gamma(m+\frac12)}{\sqrt{\pi}\Gamma(m)}\left(\frac{\pi}{4}\right)^{2m},
\end{aligned}
\]
\[
\begin{aligned}
|H_m(\frac{\omega}{8})|^2\ge1-c_m\left(\frac{\pi}{4}\right)^{2m},
\end{aligned}
\]
and
\[
\begin{aligned}
\prod_{\ell=1}^\infty|H_m(2^{-l-3}\omega)|^2&\ge\prod_{\ell=1}^\infty|1-c_m\left(2^{-l-3}\omega\right)^{2m}|
\ge
\prod_{\ell=1}^\infty|1-c_m\left(\frac{\pi}{4}\right)^{2m}(2^{-2m})^{l}|
\\&\ge
\prod_{\ell=1}^\infty|1-(2^{-2m})^{l}|\ge(1-2^{-2m})^{1/(1-2^{-2m})}.
\end{aligned}
\]
Thus,
\[
\begin{aligned}
I_1&=\frac{1}{(\sqrt{2\pi})^p}\int_{|\omega|\le\pi}|\omega|^{-mp}
\left[
|H_m(\omega/2+\pi)|^2|H_m(\omega/4)|^2|H_m(\omega/8)|^2\right.
\\&\qquad\qquad\times\left.\prod_{\ell=1}^\infty|H_m(2^{-l-3}\omega)|^2
\right]^{p/2}d\omega
\\&
\ge(1-o(1))\frac{1}{(\sqrt{2\pi})^p}\int_{|\omega|\le\pi}|\omega|^{-mp}|H_m(\omega/2+\pi)|^pd\omega.
\end{aligned}
\]
Obviously,
\[
I_1\le
\frac{1}{(\sqrt{2\pi})^p}\int_{|\omega|\le\pi}|\omega|^{-mp}|H_m(\omega/2+\pi)|^pd\omega.
\]
Now, we use the property of $H_m$ to deduce the asymptotic behavior
of
\[
I_{11}:=\int_{|\omega|\le\pi}|\omega|^{-mp}|H_m(\omega/2+\pi)|^pd\omega.
\]
Let $u=\frac{\sin^2t}{\sin^2(\omega/2)}$. We have
\[
\begin{aligned}
|H_m(\omega/2+\pi)|^2 &= c_m\int_0^{\omega/2}\sin^{2m-1}tdt
\\&=\frac{c_m}{2}\sin^{2m}(\omega/2)\int_0^1u^{m-1}(1-u\sin^2(\omega/2))^{-1/2}du
\end{aligned}
\]
Since
\[
\frac1m=\int_0^1u^{m-1}du\le\int_0^1u^{m-1}(1-u\sin^2(\omega/2))^{-1/2}du\le\int_0^1u^{m-1}(1-u)^{-1/2}du=c_m^{-1}
\]
and
\[
\begin{aligned}
 I_{11}&= 2\int_0^{\pi}|\omega|^{-mp}\cdot
\left[\frac{c_m}{2}\sin^{2m}(\omega/2)\int_0^1u^{m-1}(1-u\sin^2(\omega/2))^{-1/2}du\right]^{p/2}d\omega,
\end{aligned}
\]
we obtain
\[
 \left(\frac{c_m}{2m}\right)^{p/2}\cdot
2^{-mp}\cdot\int_0^\pi\left(\frac{\sin(\omega/2)}{\omega/2}\right)^{mp}d\omega
\le \frac12 I_{11}\le
 \left(\frac12\right)^{-p/2}\cdot
2^{-mp}\cdot\int_0^\pi\left(\frac{\sin(\omega/2)}{\omega/2}\right)^{mp}d\omega.
\]
Now by that
$\int_{-\pi}^\pi\left(\frac{\sin(\omega/2)}{\omega/2}\right)^{2\cdot
mp/2}d\omega = C (\sqrt{mp/2} )^{-1}(1+\mathcal{O}(m^{-1/2}))$ and
$\frac{1}{\pi}<\frac12$, we conclude that
\[
\|\wh{_m\psi_m^D}\|_p^p = C
\cdot\frac{2}{(\sqrt{2\pi})^p}\cdot\frac{2^{-mp}\cdot
A(m)^p}{\sqrt{mp/2}}\cdot(1+\mathcal{O}(m^{-1/2})),
\]
which completes our proof.
\end{proof}

\subsection{The  Wavelet Coefficients of Semiorthogonal Spline Wavelets}
In this subsection, we mainly focus on the asymptotic behavior of
wavelet coefficients for the semiorthogonal spline wavelets. We
shall present the asymptotic estimations of the following
quantities: $\|\wh{_k\varphi^S_m}\|_p$, $\|\wh{_k\psi^S_m}\|_p$, and
$\|\wh{_m\psi^S_m}\|_p$.

First, for the scaling function $\varphi^S_m$, which is the B-spline
$N_m$ of order $m$, we have the following result:
\begin{theorem}
\label{thm:Cphi} Let $\varphi_m^S:=N_{m}$ be the B-Spline of order
$m$. Let $k\ge0$ be an integer. Then
\begin{equation}
\|\wh{_k\varphi_m^S}\|_p=\frac{8^{1/p}}{(\sqrt{2\pi})^{1-1/p}}\cdot\frac{1}{(\sqrt{\Lambda_1
m
p})^{1/p}}\cdot(2\xi_1)^{-k}\cdot(\lambda_1/\xi_1)^{m/2}\cdot(1+\mathcal
{O}(m^{-1/2})),
\end{equation}
where
\begin{equation}
\begin{aligned}
\lambda_1 & = \frac{\sin^2(\xi_1)}{\xi_1} = 0.72461...,\\
\Lambda_1 & =
-\frac12\frac{d^2}{d\omega^2}\ln\frac{\sin^2(\xi_1-\omega)}{\xi_1-\omega}\Big{|}_{\omega=0}=
0.81597...,
\end{aligned}
\end{equation}
and $\xi_1=1.1655...$ is the unique solution of the transcendental
equation $ \xi_1-2\cot(\xi_1)=0$ in the interval $(0,\pi)$.
\end{theorem}
\begin{proof} By
$\wh{\varphi_m^S}(\omega)=\frac{1}{\sqrt{2\pi}}(e^{-i\omega/2}\frac{\sin(\omega/2)}{\omega/2})^m$,
\[
\begin{aligned}
\|\wh{_k\varphi_m^S}\|_p^p&=
\int_\R|\omega|^{-kp}\cdot\Big|\frac{\sin(\omega/2)}{\omega/2}\Big|^{mp}d\omega
=\frac{2^{1-kp}}{(\sqrt{2\pi})^p}\int_R|\omega|^{-kp}\cdot\Big|\frac{\sin(\omega)}{\omega}\Big|^{mp}d\omega
\\
&=\frac{2^{2-kp}}{(\sqrt{2\pi})^p}\int_0^\infty\omega^{-p(m/2+k)}\cdot\left(\frac{\sin^2(\omega)}{\omega}\right)^{mp/2}d\omega
\\
&=\frac{2^{2-kp}}{(\sqrt{2\pi})^p}\left\{\int_0^\pi+\int_{\pi}^\infty\right\}\omega^{-p(m/2+k)}\cdot\left(\frac{\sin^2(\omega)}{\omega}\right)^{mp/2}d\omega
\\&
=:\frac{2^{2-kp}}{(\sqrt{2\pi})^p}(I_1+I_2).
\end{aligned}
\]
For $I_2$ with $mp>1$, we have
\[
I_2\le \int_\pi^\infty\omega^{-mp}d\omega
=\frac{1}{mp-1}\left(\frac{1}{\pi}\right)^{mp-1}.
\]
To estimate $I_1$, we use the same technique as in the proof of
\cite[Lemma~4]{Ehrich:2000}. Let $\xi_1$ be the point where
$\sin^2(\omega)/\omega$ takes its maximum value $\lambda_1$ in
$(0,\pi)$, i.e., $\xi_1=1.1655...$ is the root of the transcendental
equation $ \xi_1^{-1}-2\cot(\xi_1)=0$ and
$\lambda_1=\frac{\sin^2(\xi_1)}{\xi_1}=0.72461...$ . Separate $I_1$
to two parts as follows
\[
I_1=\left\{\int_0^{\xi_1}+\int_{\xi_1}^\pi\right\}
\left(\frac{\sin^2(\omega)}{\omega}\right)^{mp/2}\cdot\omega^{-p(m/2+k)}d\omega
=:I_{11}+I_{12}.
\]
We first estimate $I_{11}$. Let
\[
t=t(\omega)=
\ln\frac{\xi_1-\omega}{\sin^2(\xi_1-\omega)}-\ln\frac{\xi_1}{\sin^2(\xi_1)}
=\ln\frac{\lambda_1(\xi_1-\omega)}{\sin^2(\xi_1-\omega)},\quad\omega\in(0,\xi_1).
\]
Then,
\[
t(\omega)\sim a_2\omega^2+a_3\omega^3+\cdots\sim
a_2\omega^2\left(1+\frac{a_3}{a_2}\omega+\cdots\right),\quad
\omega\rightarrow0,
\]
where
\[
a_2 = \Lambda_1=
-\frac12\frac{d^2}{d\omega^2}\ln\frac{\sin^2(\xi_1-\omega)}{\xi_1-\omega}\Big{|}_{\omega=0}=
0.81597...\,.
\]
Then, similar to the proof of  \cite[Lemma~4]{Ehrich:2000}, we can
obtain
\[
\begin{aligned}
\omega = \omega(t) &\sim
(\Lambda_1)^{-1/2}\sqrt{t}(1+c_1t^{1/2}+c_2t+\cdots),
\\ \frac{d\omega}{dt}&\sim
\frac{1}{2\sqrt{\Lambda_1t}}(1+d_1t^{1/2}+d_2t+\cdots),
\\
\xi_1-\omega(t)&\sim\xi_1(1-e_1t^{1/2}-e_2 t-\cdots),
\end{aligned}
\]
for $t\rightarrow0$. Changing the variable of $I_{11}$, we have
\[
\begin{aligned}
I_{11} &= \int_0^{\xi_1}
\left(\frac{\sin^2(\omega)}{\omega}\right)^{mp/2}\cdot\omega^{-p(m/2+k)}d\omega
\\&=\int_{0}^{\xi_1}\left(\frac{\sin^2(\xi_1-\omega)}{\xi_1-\omega}\right)^{mp/2}\cdot(\xi_1-\omega)^{-p(m/2+k)}d\omega
\\&=\lambda_1^{mp/2}\int_0^\infty e^{-\frac{mp}{2}t}q(t)dt,
\end{aligned}
\]
where
\[
q(t)\sim
\left(\xi_1^{p(m/2+k)}\sqrt{\Lambda_1t}\right)^{-1}(1+f_1t^{1/2}+f_2t+\cdots).
\]
Now by Watson's lemma, we have
\[
\begin{aligned}
I_{11}&=\lambda_1^{mp/2}\cdot\left(\xi_1^{p(m/2+k)}\sqrt{\Lambda_1}\right)^{-1}\cdot\frac{\sqrt{\pi}}{\sqrt{mp/2}}\cdot(1+\mathcal{O}(m^{-1/2}))
\\&=\frac{\sqrt{2\pi}}{\sqrt{\Lambda_1mp}}\cdot(\xi_1)^{-kp}\cdot\left(\frac{\lambda_1}{\xi_1}\right)^{mp/2}\cdot(1+\mathcal{O}(m^{-1/2})).
\end{aligned}
\]
For $I_{12}$, we use
\[
t=t(\omega)=
\ln\frac{\xi_1+\omega}{\sin^2(\xi_1+\omega)}-\ln\frac{\xi_1}{\sin^2(\xi_1)}
=\ln\frac{\lambda_1(\xi_1+\omega)}{\sin^2(\xi_1+\omega)},\quad\omega\in(0,\pi-\xi_1).
\]
Similarly, we have
$
%\begin{aligned}
I_{12}%&=\lambda_1^{mp/2}\cdot\left(\xi_1^{p(m/2+k)}\sqrt{\Lambda_1}\right)^{-1}\cdot\frac{\sqrt{\pi}}{\sqrt{mp/2}}\cdot(1+\mathcal{O}(m^{-1/2}))
=\frac{\sqrt{2\pi}}{\sqrt{\Lambda_1mp}}\cdot(\xi_1)^{-kp}\cdot\left(\frac{\lambda_1}{\xi_1}\right)^{mp/2}\cdot(1+\mathcal{O}(m^{-1/2})).
%\end{aligned}
$
Consequently,
$
%\begin{aligned}
I_{1}%&=\lambda_1^{mp/2}\cdot\left(\xi_1^{p(m/2+k)}\sqrt{\Lambda_1}\right)^{-1}\cdot\frac{\sqrt{\pi}}{\sqrt{mp/2}}\cdot(1+\mathcal{O}(m^{-1/2}))
=\frac{2\sqrt{2\pi}}{\sqrt{\Lambda_1mp}}\cdot(\xi_1)^{-kp}\cdot\left(\frac{\lambda_1}{\xi_1}\right)^{mp/2}\cdot(1+\mathcal{O}(m^{-1/2})).
%\end{aligned}
$
Noting that
$
\frac{1}{\pi}=0.31830...<\left(\frac{\lambda_1}{\xi_1}\right)^{1/2}=0.78846...,
$
we conclude
\[
\begin{aligned}
\|\wh{_k\varphi_m^S}\|_p^p
&=\frac{2^{2-kp}}{(\sqrt{2\pi})^p}\cdot\frac{2\sqrt{2\pi}}{\sqrt{\Lambda_1mp}}\cdot(\xi_1)^{-kp}\cdot\left(\frac{\lambda_1}{\xi_1}\right)^{mp/2}\cdot(1+\mathcal{O}(m^{-1/2}))
\\
&=\frac{8}{(\sqrt{2\pi})^{p-1}}\cdot\frac{1}{\sqrt{\Lambda_1mp}}\cdot(2\xi_1)^{-kp}\cdot\left(\frac{\lambda_1}{\xi_1}\right)^{mp/2}\cdot(1+\mathcal{O}(m^{-1/2})),
\end{aligned}
\]
which completes our proof.
\end{proof}
Next, for the spline wavelet function $\psi^S_m$, we have the
following estimation.
\begin{theorem}\label{thm:Cpsi}
Let $k\in\N\cup\{0\}$ be a fixed nonnegative integer. Let $\psi_m^S$
be the semiorthogonal spline wavelet of order $m$, i.e.,
\begin{equation}\label{def:splinewavelets}
\psi_m^S(x):=\sum_{\nu=0}^{2m-2}\frac{(-1)^\nu}{2^{m-1}}N_{2m}(v+1)N_{2m}^{(m)}(2x-\nu),\quad
x\in\R.
\end{equation}
Then
\begin{equation}
\|\wh{_k\psi_m^S}\|_p=\frac{2^{3/p}}{(\sqrt{2\pi})^{1-1/p}}\cdot\frac{(2\pi-4\xi_2)^{-k}}{(\sqrt{2\Lambda_2mp})^{1/p}}\cdot\lambda_2^m\cdot(1+\mathcal{O}(m^{-1/2})).
\end{equation}
where
\begin{equation}
\begin{aligned}
\lambda_2 &=
\frac{\sin^2(\xi_2-\pi/2)\sin^2(\xi_2)}{(\pi/2-\xi_2)\xi_2^2}=0.69706...\\
\Lambda_2&=
-\frac12\frac{d^2}{du^2}\ln\frac{\sin^2(u-\pi/2)\sin^2(u)}{(\pi/2-u)u^2}\Big{|}_{u=\xi_2}=1.2229...,
\end{aligned}
\end{equation}
and $\xi_2=0.2853...$ is the unique solution of the transcendental
equation
\[
(2\pi\xi_2-4\xi_2^2)\cos(2\xi_2)+(3\xi_2-\pi)\sin(2\xi_2)=0,\quad
\xi\in(0,\pi/2).
\]
\end{theorem}
\begin{proof}
Using the Fourier transform of the B-spline and the definition of
Euler-Frobenius polynomial $E_{2m-1}(z)$ for $z=e^{i\omega}$:
\[
\begin{aligned}
\frac{E_{2m-1}(z)}{(2m-1)!}&=\sum_{\nu=0}^{2m-2}N_{2m}(\nu+1)z^\nu
=e^{-i(m-1)\omega}(2\sin(\omega/2))^{2m}\sum_{l=-\infty}^\infty\frac{1}{(\omega+2\pi
l)^{2m}},
\end{aligned}
\]
we can derive (c.f. \cite[Lemma~4]{Ehrich:2000})
\[
\begin{aligned}
|\wh{_k\psi_m^S}(\omega)|
&=\frac{2^{-2k}}{\sqrt{2\pi}}\Big|\frac{\sin^2(\omega/4)}{\omega/4}\Big|^m\Big|\frac{\omega}{4}\Big|^{-k}
\Big|\frac{E_{2m-1}(\wt{z})}{(2m-1)!}\Big|
\\&=\frac{2^{-2k}}{\sqrt{2\pi}}\Big|\frac{\sin^2(\omega/4)}{\omega/4}\Big|^m\Big|\frac{\omega}{4}\Big|^{-k}
\Big|2\sin(\wt\omega/2)\Big|^{2m}\Big|\sum_{\ell=-\infty}^\infty\frac{1}{(\wt\omega+2\pi
\ell)^{2m}}\Big|,
\end{aligned}
\]
where $\wt z = e^{i\wt\omega}$ and $\wt\omega=\pi-\omega/2$. Then,
\[
\begin{aligned}
\|\wh{_k\psi_m^S}\|_p^p& = \frac{2^{-2kp}}{(\sqrt{2\pi})^p}\int_\R
\Big|\frac{\sin^2(\omega/4)}{\omega/4}\Big|^{mp}\Big|\frac{\omega}{4}\Big|^{-kp}
\Big|2\sin(\wt\omega/2)\Big|^{2mp}\Big|\sum_{l=-\infty}^\infty\frac{1}{(\wt\omega+2\pi
l)^{2m}}\Big|^pd\omega\\
%&= \frac{2^{-2kp}}{(\sqrt{2\pi})^p}\int_\R\left[
%\left(\frac{\sin^2(\omega/4)}{\omega/4}\right)^{2
%m}\left(\frac{\omega}{4}\right)^{-2k}
%\left(2\sin(\wt\omega/2)\right)^{4m}\right.
%\\&\qquad\qquad\quad\times\left.
%\left(\sum_{l=-\infty}^\infty\frac{1}{(\wt\omega+2\pi
%l)^{2m}}\right)^2\right]^{p/2}d\omega
%\\
&= \frac{2^{-2kp}}{(\sqrt{2\pi})^p}\int_\R\left[
\left(\frac{\sin^2(u-\pi/2)}{u-\pi/2}\right)^{2
m}\left(u-\pi/2\right)^{-2k} \left(2\sin(u)\right)^{4m}\right.
\\&\qquad\qquad\quad\times\left.
\left(\sum_{l=-\infty}^\infty\frac{1}{(2u+2\pi
l)^{2m}}\right)^2\right]^{p/2}4du
\\
&=
 \frac{4\cdot2^{-2kp}}{(\sqrt{2\pi})^p}\left\{\int_{-\infty}^{-\pi/2}+\int_{-\pi/2}^{\xi_2}+\int_{\xi_2}^{\pi/2}+\int_{\pi/2}^{3\pi/2}+\int_{3\pi/2}^\infty\right\}
\left[ \left(\frac{\sin^2(u-\pi/2)}{u-\pi/2}\right)^{2
m}\right.\\&\qquad\qquad\quad\times\left.\left(u-\pi/2\right)^{-2k}
\left(\sin(u)\right)^{4m}
\left(\sum_{\ell=-\infty}^\infty\frac{1}{(u+\pi
\ell)^{2m}}\right)^2\right]^{p/2}du \\
& =:
 \frac{4\cdot2^{-2kp}}{(\sqrt{2\pi})^p}(I_1+I_2+I_3+I_4+I_5),
\end{aligned}
\]
Here, $\xi_2$ is the point where the function
\[
g(u):=\frac{\sin^2(u-\pi/2)\sin^2(u)}{(\pi/2-u)u^2}
\] takes its
maximum value in $(0,\pi/2)$, i.e., $\xi_2=0.28532...$ is the root
of the transcendental equation
\[
h(u):=(2\pi u-4u^2)\cos(2u)+(3u-\pi)\sin(2u).
\]
Note that $g'(u)=\frac{\sin(2u)}{4(\pi/2-u)^2u^4}\cdot h(u)$ and
$\lambda_2=g(\xi_2)=0.69706...$.

We first estimate $I_2$. By \cite[Lemma~3]{Ehrich:2000}, we have
\[
\begin{aligned}
I_2 =
\int_{-\pi/2}^{\xi_2}[g(u)^{2m}(u-\pi/2)^{-2k}(1+R_1+r(u))^2]^{p/2}du=:I_{21}+\wt{R},
\end{aligned}
\]
where $|R_1|\le(2m-1)^{-1}$,
\[
r(u)=
\begin{cases}
         \left(\frac{u}{\pi+u}\right)^{2m}, & -\pi/2<u\le0,\\
          \left(\frac{u}{\pi-u}\right)^{2m}, &0\le u <\xi_2.
 \end{cases}
%
%\left\{
%        \begin{array}{ll}
%          \left(\frac{u}{\pi+u}\right)^{2m}, & \hbox{$-\pi/2<u\le0$,} \\
%          \left(\frac{u}{\pi-u}\right)^{2m}, & \hbox{$0\le u <\xi_2$.}
%        \end{array}
%      \right.
\]
\[
I_{21}:=\int_{-\pi/2}^{\xi_2}[g(u)^{2m}(u-\pi/2)^{-2k}(1+R_2(u))^2]^{p/2}du,
\]
where
\[
R_2(u)= \left\{
        \begin{array}{ll}
          R_1+r(u), & \hbox{$-\pi/2+\delta<u<\xi_2$,} \\
          R_1, & \hbox{$-\pi/2< u <-\pi/2+\delta$.}
        \end{array}
      \right.
\]
$0<\delta<\pi/2-\xi_2$ is fixed. Hence
\[
|R_2(u)|\le
\frac{1}{2m-1}+\left(\frac{\pi/2-\delta}{\pi/2+\delta}\right)^{2m},
\]
and
\[
\begin{aligned}
\wt{R} &=
\int_{-\pi/2}^{-\pi/2+\delta}[g(u)^{2m}(u-\pi/2)^{-2k}]^{p/2}\cdot[(1+R_1+r(u))^p-(1+R_1)^p]du.
\\
&\le (p2^p+o(1))
\int_{-\pi/2}^{-\pi/2+\delta}[g(u)^{2m}(u-\pi/2)^{-2k}]^{p/2}du \\
&\le(p2^p+o(1))\delta\cdot\left[\frac{\sin^{4m}\delta}{(\pi-\delta)^{2m+2k}}\right]^{p/2}
%\\
\le(p2^p+o(1))\delta\cdot\frac{\sin^{2mp}\delta}{(\pi-\delta)^{p(m+k)}}.
\end{aligned}
\]
For the estimation of $I_{21}$, we shall employ the Watson's lemma.
We introduce
\[
t=t(v):=\ln g(\xi_2)-\ln
g(\xi_2-v)=\ln\frac{\lambda_2}{g(\xi_2-v)},\quad\frac{dt}{dv}=\frac{g'(\xi-v)}{g(\xi-v)},
\]
for $v\in[0,\pi/2+\xi_2]$. We have $t\rightarrow0$ as
$v\rightarrow0$ and $t$ goes from $0$ to $\infty$ monotonically as
$v$ increases from $0$ to $\pi/2+\xi_2$. We can state the asymptotic
expansion of  $t(v)$ near $v=0$ as follows:
\[
t(v)\sim a_2v^2+a_3v^3+\cdots\sim a_2v^2(1+a_3/a_2 v+\cdots),
\]
where
\[
a_2 = \Lambda_2 = -\frac12\frac{d^2}{dv^2}\ln
g(\xi_2-v)\Big|_{v=0}=-\frac{h'(\xi_2)}{2\xi_2(\pi/2-\xi_2)\sin(2\xi_2)}=1.2229....
\]
Let $s=\sqrt{t}$. Then
\[
s(v)\sim\sqrt{\Lambda_2}v(1+b_1v+\cdots), \quad v\rightarrow0.
\]
Now $s'(v)\neq0$, we can reverse this expansion,
\[
v=v(t)\sim \Lambda_2^{-1/2}s(1+c_1s+c_2s^2+\cdots)\sim
\Lambda_2^{-1/2}t^{1/2}(1+c_1t^{-1/2}+c_2t+\cdots).
\]
Also,
\[
\frac{dv}{dt} =
\frac{(\pi/2+v-\xi_2)(\xi_2-v)\sin2(\xi_2-v)}{h(\xi_2-v)}
\]
Asymptotic expansion of numerator and denominator at $v=0$ and
division yields
\[
\begin{aligned}
\frac{dv}{dt} &\sim
\frac{(\pi/2-\xi_2)\xi\sin(2\xi)}{-h'(\xi_2)v(t)}(1+d_1v(t)^2+\cdots)
\\
&\sim \frac{1}{2\Lambda_2v(t)}(1+d_1 v(t)^2+\cdots)
\\
&\sim \frac{1}{2\sqrt{\Lambda_2t}}(1+ e_1t^{1/2}+e_2t+\cdots).
\end{aligned}
\]
Now changing the variable in $I_{21}$ and noting
$g(\xi_2-v)=\lambda_2e^{-t}$, we have
\[
\begin{aligned}
I_{21}&\sim \int_{-\pi/2}^{\xi_2}[g(u)^{2m}(u-\pi/2)^{-2k}]^{p/2}du
\\&=\lambda_2^{mp}\int_{0}^{\xi_2+\pi/2}[(g(\xi_2-v)/\lambda_2)^{2m}(\xi_2-v-\pi/2)^{-2k}]^{p/2}dv
\\&=\lambda_2^{mp}\int_0^\infty e^{-mpt}q(t)dt,
\end{aligned}
\]
where
\[
\begin{aligned}
q(t)&=(\pi/2+v(t)-\xi_2)^{-kp}\cdot\frac{dv}{dt}
\\
&\sim
\frac{(\pi/2-\xi_2)^{-kp}}{2\sqrt{\Lambda_2t}}(1+f_1t^{1/2}+f_2t+\cdots)^{-kp}(1+e_1t^{1/2}+e_2t+\cdots)\\
&\sim
\frac{(\pi/2-\xi_2)^{-kp}}{2\sqrt{\Lambda_2t}}(1+g_1t^{1/2}+g_2t+\cdots).
\end{aligned}
\]
By Watson's lemma  and choosing $\delta$ such that
$\sin^2\delta/(\pi-\delta)<\lambda_2$, we conclude that
\[
I_2\sim I_{21}\sim
\lambda_2^{mp}\cdot\frac{(\pi/2-\xi_2)^{-kp}}{2\sqrt{\Lambda_2}}\cdot\frac{\sqrt{\pi}}{\sqrt{mp}}\cdot(1+\mathcal{O}(m^{-1/2})).
\]
Similarly, we can estimate the asymptotic behavior of $I_3$. We use
\[
t=t(v)=\ln g(\xi_2)-\ln g(\xi+v) = \ln
\frac{\lambda_2}{g(\xi+v)},\quad v\in(0,\pi/2-\xi_2).
\]
Same technique implies
\[
I_3\sim
\lambda_2^{mp}\cdot\frac{(\pi/2-\xi_2)^{-kp}}{2\sqrt{\Lambda_2}}\cdot\frac{\sqrt{\pi}}{\sqrt{mp}}\cdot(1+\mathcal{O}(m^{-1/2})).
\]

Next, for $I_4$, observing the period of
$\sum_{l=-\infty}^\infty\frac{1}{(u+\pi l)^{2m}}$ is $\pi$, we have
\[
\begin{aligned}
I_4 &\;\;\,= \int_{\pi/2}^{3\pi/2} \left[
\left(\frac{\sin^2(u-\pi/2)\sin^2(u)}{u-\pi/2}\right)^{2
m}\left(u-\pi/2\right)^{-2k}
\left(\sum_{\ell=-\infty}^\infty\frac{1}{(u+\pi
\ell)^{2m}}\right)^2\right]^{p/2}du
\\
& \stackrel{{u\rightarrow \pi-u}}{=} \int_{-\pi/2}^{\pi/2} \left[
\left(\frac{\sin^2(u-\pi/2)\sin^2(u)}{u-\pi/2}\right)^{2
m}\left(u-\pi/2\right)^{-2k}
\left(\sum_{\ell=-\infty}^\infty\frac{1}{(u+\pi
\ell)^{2m}}\right)^2\right]^{p/2}du
\\&\;\;\; = I_2+I_3.
\end{aligned}
\]
Consequently,
\[
I_4 \sim
2\lambda_2^{mp}\cdot\frac{(\pi/2-\xi_2)^{-kp}}{2\sqrt{\Lambda_2}}\cdot\frac{\sqrt{\pi}}{\sqrt{mp}}\cdot(1+\mathcal{O}(m^{-1/2})).
\]
Next, we estimate $I_5$. By $
E_{2m-1}(z)=(2m-1)!\sum_{\nu=0}^{2m-2}N_{2m}(\nu+1)z^\nu$, we derive
that $|E_{2m-1}(z)|\le(2m-1)!$ for $|z|=1$ and
\[
\begin{aligned}
I_5 &= \int_{3\pi/2}^\infty
 \left[
\left(\frac{\sin^2(u-\pi/2)}{u-\pi/2}\right)^{2
m}\left(u-\pi/2\right)^{-2k}
\frac{|E_{2m-1}(e^{2iu})|}{(2m-1)!}\right]^{p/2}du
\\
&\le \int_{\pi}^\infty
 \left[
\left(\frac{\sin^2(u)}{u}\right)^{2 m}u^{-2k} \right]^{p/2}du
\le
\frac{1}{\pi^{2k}}\int_{\pi}^\infty u^{-mp}du
\\&\le
\frac{1}{(mp-1)\pi^{2k}}\left(\frac{1}{\pi}\right)^{mp-1},\quad
mp-1>0.
\end{aligned}
\]
Similarly,
\[
\begin{aligned}
I_1 &= \int_{-\infty}^{-\pi/2}
 \left[
\left(\frac{\sin^2(u-\pi/2)}{u-\pi/2}\right)^{2
m}\left(u-\pi/2\right)^{-2k}
\frac{|E_{2m-1}(e^{2iu})|}{(2m-1)!}\right]^{p/2}du
\\
&\le \int_{-\infty}^{-\pi}
 \left[
\left(\frac{\sin^2(u)}{u}\right)^{2 m}u^{-2k} \right]^{p/2}du
%\\&
%\le \frac{1}{\pi^{2k}}\int_{\pi}^\infty u^{-mp}du
%\\&
\le
\frac{1}{(mp-1)\pi^{2k}}\left(\frac{1}{\pi}\right)^{mp-1},\quad
mp-1>0.
\end{aligned}
\]
In summary, we have
\[
I_1\sim I_5\le
\frac{1}{(mp-1)\pi^{2k}}\left(\frac{1}{\pi}\right)^{mp-1}
\]
and
\[
I_2\sim I_3\sim \frac12 I_4 \sim
\lambda_2^{mp}\cdot\frac{(\pi/2-\xi_2)^{-kp}}{2\sqrt{\Lambda_2}}\cdot\frac{\sqrt{\pi}}{\sqrt{mp}}\cdot(1+\mathcal{O}(m^{-1/2})).
\]
Due to $\frac{1}{\pi}=0.31830...\le \lambda_2 = 0.69706...$, we
conclude that
\[
\begin{aligned}
\|\wh{_k\psi_m^S}\|_p^p &=
\frac{4\cdot2^{-2kp}}{\sqrt{2\pi}^{p}}\cdot 4
\lambda_2^{mp}\cdot\frac{(\pi/2-\xi_2)^{-kp}}{2\sqrt{\Lambda_2}}\cdot\frac{\sqrt{\pi}}{\sqrt{mp}}\cdot(1+\mathcal{O}(m^{-1/2}))
\\& = \frac{8}{\sqrt{2\pi}^{p-1}}\cdot\frac{(2\pi-4\xi_2)^{-kp}}{\sqrt{2\Lambda_2mp}}\cdot
\lambda_2^{mp}\cdot(1+\mathcal{O}(m^{-1/2})),
\end{aligned}
\]
which completes our proof.
\end{proof}

Finally, when $k=m$, we obtain the following estimation.
\begin{theorem}\label{thm:CmPsimS}
Let $\psi_m^S$ be the spline wavelet defined in
\eqref{def:splinewavelets}. Then
\begin{equation}\label{eq:mPsimS}
\|\wh{_m\psi_m^S}\|_p =
\frac{2^{1/p}}{(\sqrt{2\pi})^{1-1/p}}\cdot\left(\frac{\pi}{\sqrt{\pi^2-8}}\right)^{1/p}\cdot\frac{1}{(\sqrt{2mp})^{1/p}}\cdot\left(\frac{16}{\pi^4}\right)^{m}\cdot(1+\mathcal{O}(m^{-1/2})).
\end{equation}
\end{theorem}
\begin{proof}
By definition, $_m\psi_m^S(x) =
\sum_{\nu=0}^{2m-2}\frac{(-1)^\nu}{2^{m-1}}N_{2m}(\nu+1)N_{2m}(2x-\nu)$.
Hence,
\[
\begin{aligned}
|\wh{_m\psi_m^S}(\omega)|
&=\frac{2^{-2m}}{\sqrt{2\pi}}\left(\frac{\sin(\omega/4)}{\omega/4}\right)^{2m}\frac{|E_{2m-1}(\wt
z)|}{(2m-1)!}
\\
&
=\frac{2^{-2m}}{\sqrt{2\pi}}\left(\frac{\sin(\omega/4)}{\omega/4}\right)^{2m}(2\sin(\wt\omega/2))^{2m}\left|\sum_{l=-\infty}^\infty\frac{1}{(\wt\omega+2\pi
l)^{2m}}\right|,
\end{aligned}
\]
where $E_{2m-1}$ is the Euler-Frobenius polynomial, $\wt z =
e^{i\wt\omega}$, and $\wt\omega= \pi-\omega/2$. Setting $u
=\wt\omega/2 = \pi/2-\omega/4$, we obtain
\[
\begin{aligned}
\|\wh{_m\psi_m^S}\|_p^p &= \frac{4\cdot
2^{-2mp}}{(\sqrt{2\pi})^p}\int_\R\left[\left(\frac{\sin(u-\pi/2)}{u-\pi/2}\right)^{4m}
\cdot(\sin(u))^{4m}\cdot \left(
\sum_{\ell=-\infty}^\infty\frac{1}{(u+\pi
\ell)^{2m}}\right)^2\right]^{p/2}du
\\& =
 \frac{4\cdot
2^{-2mp}}{(\sqrt{2\pi})^p}\left\{\int_{-\infty}^{-\pi/2}+\int_{-\pi/2}^{\pi/4}+\int_{\pi/4}^{\pi}+\int_{\pi}^{\infty}\right\}
\left[\left(\frac{\sin(u-\pi/2)\sin(u)}{u-\pi/2}\right)^{4m}\right.
\\&
\qquad\qquad\qquad\times\left.\left(
\sum_{\ell=-\infty}^\infty\frac{1}{(u+\pi
\ell)^{2m}}\right)^2\right]^{p/2}du
\\& =:
 \frac{4\cdot
2^{-2mp}}{(\sqrt{2\pi})^p}(I_1+I_2+I_3+I_4).
\end{aligned}
\]
Let
\[
g(u):=\left(\frac{\sin(u-\pi/2)\sin(u)}{(u-\pi/2)u}\right)^2.
\]
Then $g$ is symmetric about $u=\pi/4$ and $g(u)\le g(\pi/4) =
64/\pi^4$. Similarly, using \cite[Lemma~3]{Ehrich:2000}, we have
\[
I_2\sim \int_{-\pi/2}^{\pi/4}(g(u))^{mp}du
\]
Introducing
\[
t=t(v)=\ln\frac{g(\pi/4)}{g(\pi/4-v)}, \quad v\in[0,\frac34\pi],
\]
we can derive
\[
q(t):=\frac{dv}{dt}\sim
\frac{\pi}{4}(\pi^2-8)^{-1/2}t^{-1/2}(1+e_1t^{1/2}+e_2t+\cdots).
\]
Changing the variable $u\rightarrow \pi/4-v$ in $I_2$ and using
Watson's lemma, we deduce
\[
\begin{aligned}
 \frac{4\cdot
2^{-2mp}}{(\sqrt{2\pi})^p}I_2 &\sim \frac{4\cdot
2^{-2mp}}{(\sqrt{2\pi})^p}[g(\pi/4)]^{mp}\int_0^\infty e^{-mpt}
q(t)dt \\&\sim
\frac{1}{\sqrt{2\pi}^{p-1}}\cdot\frac{\pi}{\sqrt{\pi^2-8}}
\left(\frac{16}{\pi^4}\right)^{mp}\cdot\frac{1}{\sqrt{2mp}}\cdot(1+\mathcal{O}(m^{-1/2}))
\end{aligned}
\]
It is easily seen that $I_3=I_2$  due to the symmetry of $g(u)$.
Also, by the symmetry, we have $I_1=I_4$. Using the fact that
$|E_{2m-1}(z)|\le (2m-1)!$ for $|z| = 1$, we have
\[
\begin{aligned}
 \frac{4\cdot
2^{-2mp}}{(\sqrt{2\pi})^p}I_4&=
 \frac{
2^{-2mp}}{(\sqrt{2\pi})^p}\int_{\pi}^{\infty}\left(\frac{\sin(u-\pi/2)}{u-\pi/2}\right)^{2mp}\left(\frac{|E_{2m-1}(\wt
z)|}{(2m-1)!}\right)^pd\omega
\\
&\le \frac{ 2^{-2mp}}{(\sqrt{2\pi})^p}
\int_{\pi}^{\infty}\left(\frac{\sin(u-\pi/2)}{u-\pi/2}\right)^{2mp}d\omega
\le\frac{ 2^{-2mp}}{(\sqrt{2\pi})^p}
\int_{\pi/2}^{\infty}\left(\frac{1}{\omega}\right)^{2mp}d\omega
\\&\le
\frac{
2^{-2mp}}{(\sqrt{2\pi})^p}\frac{1}{2mp-1}\left(\frac{2}{\pi}\right)^{2mp-1}
 =\frac{1
}{2mp-1}\frac{1}{(\sqrt{2\pi})^{p-2}}\left(\frac{1}{\pi^2}\right)^{mp}.
\end{aligned}
\]
Noting that $1/\pi^2\le 16/\pi^4$, we conclude that
\[
\begin{aligned}
 \|\wh{_m\psi_m^S}\|_p^p =
\frac{2}{\sqrt{2\pi}^{p-1}}\cdot\frac{\pi}{\sqrt{\pi^2-8}}
\cdot\frac{1}{\sqrt{2mp}}\cdot\left(\frac{16}{\pi^4}\right)^{mp}\cdot(1+\mathcal{O}(m^{-1/2})),
\end{aligned}
\]
which completes our proof.
\end{proof}

\subsection{Comparison of Daubechies orthormal Wavelets and Semiorthogonal Wavelets}
Now, by the results we obtained in the above two subsections, we can
compare the Daubechies orthonormal wavelets and the semiorthogonal
spline wavelets using the constants $C_{k,p}(f)$. Note that both
Daubechies orthonormal wavelets and the semiorthogonal spline
wavelets have the same support length and number of vanishing
moments, thereby a comparison is possible in this respect.

We first consider the situation when $k$ is fixed and let
$m\rightarrow\infty$. For Daubechies family, by
Theorem~\ref{thm:CphiD} and \eqref{eq:kPsiD}, we can deduces the
following result.

\begin{corollary}\label{cor:CPhiPsiD}
Let $\varphi^D_m$ and $\psi^D_m$ be the Daubechies orthonormal
scaling function and wavelet function of order $m$, respectively.
Let $k\ge0$ be a nonnegative integer.  Then
\begin{equation}
\label{eq:CphiD}
\lim_{m\rightarrow\infty}C_{-k,p}(\varphi_m^D)=\lim_{m\rightarrow\infty}
\frac{\|\wh{_{-k}\varphi_m^D}\|_p}{\|{\wh{\varphi_m^D}}\|_p}=\frac{\pi^{k}}{(1+pk)^{1/p}}
\end{equation}
and
\begin{equation}
\label{eq:CpsiD}
\lim_{m\rightarrow\infty}C_{k,p}(\psi_m^D)=\lim_{m\rightarrow\infty}
\frac{\|\wh{_k\psi_m^D}\|_p}{\|{\wh{\psi_m^D}}\|_p}=\pi^{-k}\left(\frac{1-2^{1-pk}}{pk-1}\right)^{1/p}.
\end{equation}
\end{corollary}

For the semiorthogonal spline wavelet family, by
Theorems~\ref{thm:Cphi} and \ref{thm:Cpsi}, we can deduces the
following result.

\begin{corollary}\label{cor:CPhiPsiS}
Let $\varphi^S_m$ and $\psi^S_m$ be the semiorthogonal spline
wavelet of order $m$, respectively. Let $k\ge0$ be an integer.  Then
\begin{equation}
\label{eq:CphiS}
\lim_{m\rightarrow\infty}C_{k,p}(\varphi_m^S)=\lim_{m\rightarrow\infty}
\frac{\|\wh{_k\varphi_m^S}\|_p}{\|{\wh{\varphi_m^S}}\|_p}=(2\xi_1)^{-k}=(2.331...)^{-k}
\end{equation}
and
\begin{equation}
\label{eq:CpsiS}
\lim_{m\rightarrow\infty}C_{k,p}(\psi_m^S)=\lim_{m\rightarrow\infty}
\frac{\|\wh{_k\psi_m^S}\|_p}{\|{\wh{\psi_m^S}}\|_p}=(2\pi-4\xi_2)^{-k}=(5.1419...)^{-k},
\end{equation}
where $\xi_1=1.1655...,\xi_2=0.2853...$ are constants given in
Theorems~\ref{thm:Cphi} and ~\ref{thm:Cpsi}.
\end{corollary}

Comparing Corollarys~\ref{cor:CPhiPsiD} and \ref{cor:CPhiPsiS}, we
obtain that for every $k\in\N\cup\{0\}$, the semiorthogonal spline
wavelets are better than the Daubechies orthonormal wavelets in the
sense of asymptotically smaller constants. More precisely, we have
\begin{corollary}\label{cor:CpsiDandCpsiS}
Let $\psi^D_m$ and $\psi^S_m$ be Daubechies orthonormal wavelet  the
semiorthogonal spline wavelet of order $m$, respectively. Then
\begin{equation}
\label{eq:CpsiS.CpsiD}
\lim_{k\rightarrow\infty}\lim_{m\rightarrow\infty}\left(\frac{C_{k,p}(\psi_m^S)}{C_{k,p}(\psi_m^D)}\right)^{1/k}=\frac{\pi}{2\pi-4\xi_2}=0.61098...
.
\end{equation}
\end{corollary}
That is, the semiorthogonal spline wavelet constant $C_k(\psi^S_m)$
is exponentially better than Daubechies orthonormal wavelet constant
$C_k(\psi^D_m)$ for increasing $k$.

Since the number of vanishing moments increases with $m$, it is
natural to consider the behavior of the constants $C_k(\psi^S_m)$
with $k=k(m)=m$. In this situation, from Theorems~\ref{thm:CmpsimD}
and \ref{thm:CmPsimS}, we have the following result, which shows
that for smooth functions, the ration in \eqref{eq:CpsiS.CpsiD} when
$k=m$ is even more in favor of the spline wavelets.
\begin{corollary}\label{cor:CmpsimDandCmpsimS}
Let $\psi^D_m$ and $\psi^S_m$ be Daubechies orthonormal wavelet  the
semiorthogonal spline wavelet of order $m$, respectively. Then
\begin{equation}
\label{eq:Cmpsim}
\lim_{m\rightarrow\infty}\left(C_{m,p}(\psi_m^D)\right)^{1/m}=\frac12,\quad
\lim_{m\rightarrow\infty}\left(C_{m,p}(\psi_m^S)\right)^{1/m}=\frac{16}{\lambda_2\pi^4},
\end{equation}
and
\begin{equation}
\label{eq:CmpsimS.CmpsimD}
\lim_{m\rightarrow\infty}\left(\frac{C_{m,p}(\psi_m^S)}{C_{m,p}(\psi_m^D)}\right)^{1/m}=\frac{32}{\lambda_2\pi^4}=0.47128...
.
\end{equation}
\end{corollary}

In order to study the high dimensional tensor product wavelets, we
need to compare the asymptotic behaviors  between the scaling
function $\varphi$ and the wavelet function $\psi$ for both the
Daubechies orthonormal wavelets and seim-orthonormal spline
wavelets.

For the Daubechies orthonormal wavelets, again, by
Theorems~\ref{thm:CphiD} and \eqref{eq:kPsiD}, we have the following
result.
\begin{corollary}\label{cor:CphiDCpsiDRatio}
Let $\varphi^D_m$ and $\psi^D_m$ be the Daubechies orthonormal
scaling function and wavelet function of order $m$, respectively.
Let $k_1,k_2\ge0$ be nonnegative integers.  Then
\begin{equation}
\label{eq:CphiDCPsiDRatio}
\lim_{m\rightarrow\infty}\frac{\|\wh{_{-k_1}\varphi_m^D}\|_p}{\|{\wh{_{k_2}\psi_m^D}}\|_p}
=\frac{\pi^{k_1+k_2}}{(1-2^{1-pk_2})^{1/p}}\left(\frac{pk_1+1}{pk_2-1}\right)^{1/p}.
\end{equation}
\end{corollary}
For the semiorthogonal wavelets, similarly, using the results of
Theorems~\ref{thm:Cphi} and \ref{thm:Cpsi}, we have
\begin{corollary}\label{cor:CphiDCpsiDRatio}
Let $\varphi^S_m$ and $\psi^S_m$ be the semiorthogonal spline
scaling function and wavelet function of order $m$, respectively.
Let $k_1,k_2\ge0$ be nonnegative integers.  Then
\begin{equation}
\label{eq:CphiSCPsiSRatio}
\lim_{m\rightarrow\infty}\left(\frac{\|\wh{_{k_1}\varphi_m^S}\|_p}{\|{\wh{_{k_2}\psi_m^S}}\|_p}\right)^{1/m}
=\sqrt{\frac{\lambda_1}{\xi_1\lambda_2^2}}=1.1311....
\end{equation}
\end{corollary}

%
%
%
%
%
%
%
%
%=================HIGHER-DIMENSIONAL CASE================%

\section{High-dimensional Wavelet coeifficients}

One of the simplest way to construct high-dimensional wavelets is
using tensor product. In this section, we discuss the wavelet
coefficients for high-dimensional tensor product wavelets. We shall
mainly focus on dimension two while results of high dimensions can
be similarly obtained due to the properties of tensor product.

Let $\varphi,\psi$ be the refinable function and wavelet function
that generates an wavelet basis in $L_2(\R)$. Then, in
two-dimensional case, the refinable  function
$\fPhi(x_1,x_2)=\varphi(x_1)\varphi(x_2)$ and we have three wavelets
instead of one,
\begin{equation}\label{def:WaveletHighD}
\begin{array}{ll}
{\fPsi}^{1}(x_1,x_2):=\psi(x_1)\varphi(x_2),\\
{\fPsi}^{2}(x_1,x_2):=\varphi(x_1)\psi(x_2),\\
{\fPsi}^{3}(x_1,x_2):=\psi(x_1)\psi(x_2).
\end{array}
\end{equation}

Let $k=(k_1,k_2)\in\Z^2$ be a two-dimensional index. Then, for a
two-dimensional wavelet function $\fPsi$, we can define
$C_{k,p}(\fPsi)$ similar to \eqref{def:Ckp} by
\begin{equation}\label{def:CkpHighD}
C_{k,p}(\fPsi)=\sup_{f\in \mathcal{A}_k^{p'}}\frac{|\la
f,\fPsi\ra|}{\|\wh\fPsi\|_p}=\frac{\|\wh{_k\fPsi}\|_p}{\|\wh{\fPsi}\|_p},
\end{equation}
where $1< p,p'< \infty$, $1/p'+1/p=1$ and
$\mathcal{A}_k^{p'}:=\{f\in L_{p'}(\R^2): \|(\iu
\omega)^k\wh{f}(\omega)\|_{p'}\le 1\}$. Here, for
$x=(x_1,x_2)\in\R^2, k=(k_1,k_2)\in\Z^2$, $x^k:=x_1^{k_1}x_2^{k_2}$.
And for a function $f\in L_1(\R^2)$, $_kf$ is defined to be a
function such that $\wh{_kf}(\omega)=(i\omega)^k\wh{f}$, where
$\omega=(\omega_1,\omega_2)\in\R^2$. In particular, when
$\fPsi(x_1,x_2)=\psi_1(x_1)\psi_2(x_2)$, one can easily show that
$C_{k,p}(\fPsi)=C_{k_1,p}(\psi_1)C_{k_2,p}(\psi_2)$.
%Also, it is
%easily seen that
%\[
%\begin{aligned}
%\widehat{_k{\fPsi}^{1}}(\omega_1,\omega_2)&=\widehat{_{k_1}\psi}(\omega_1)\widehat{_{k_2}\varphi}(\omega_2),
%\\
%\widehat{_k{\fPsi}^{2}}(x_1,x_2)&=\widehat{_{k_1}\varphi}(\omega_1)\widehat{_{k_2}\psi}(\omega_2),
%\\
%\widehat{_k{\fPsi}^{3}}(x_1,x_2)&=\widehat{_{k_1}\psi}(\omega_1)\widehat{_{k_2}\psi}(\omega_2).
%\end{aligned}
%\]

In two-dimensional case, the semiorthogonal wavelets can be
represented by
\begin{equation}\label{def:2DsplineWavelets}
\begin{array}{ll}
{\fPsi}^{S,1}_{m}(x_1,x_2):=\psi_m^S(x_1)\varphi_m^S(x_2)=\psi_m^S(x_1)N_m(x_2),\\
{\fPsi}^{S,2}_{m}(x_1,x_2):=\varphi_m^S(x_1)\psi_m^S(x_2)=N_m(x_1)\psi_m^S(x_2),\\
{\fPsi}^{S,3}_{m}(x_1,x_2):=\psi_m^S(x_1)\psi_m^S(x_2).
\end{array}
\end{equation}
We can obtain that following corollaries using results in previous
sections and the properties of tensor product.
\begin{corollary}\label{cor:2DSpline}
Let ${\fPsi}^{S,1}_{m}$, ${\fPsi}^{S,2}_{m}$, and
${\fPsi}^{S,3}_{m}$ be defined in \eqref{def:2DsplineWavelets}. Let
$k=(k_1,k_2)\in\N^2$. Then,
\begin{equation}
\begin{aligned}
C_{k,p}({\fPsi}^{S,1}_{m})&\ge
\frac{1}{2^{k_1}}\left(\frac{1}{\pi}\right)^{k_1+k_2}\frac{\sqrt{K_{2(m+k_1)+1}K_{2(m+k_2)+1}}}{K_{2m+1}},
\\C_{k,p}({\fPsi}^{S,2}_{m})&\ge
\frac{1}{2^{k_2}}\left(\frac{1}{\pi}\right)^{k_1+k_2}\frac{\sqrt{K_{2(m+k_1)+1}K_{2(m+k_2)+1}}}{K_{2m+1}},
\\C_{k,p}({\fPsi}^{S,3}_{m})&\ge
\left(\frac{1}{2\pi}\right)^{k_1+k_2}\frac{\sqrt{K_{2(m+k_1)+1}K_{2(m+k_2)+1}}}{K_{2m+1}},
\end{aligned}
\end{equation}
where $K_j$'s are the Favard's constants. And
\begin{equation}\label{def:2DDaub}
\begin{aligned}
\lim_{m\rightarrow\infty}C_{k,p}({\fPsi}^{S,1}_{m})&=(2\pi-4\xi_2)^{-k_1}(2\xi_1)^{-k_2},\\
\lim_{m\rightarrow\infty}C_{k,p}({\fPsi}^{S,2}_{m})&=(2\pi-4\xi_2)^{-k_2}(2\xi_1)^{-k_1},\\
\lim_{m\rightarrow\infty}C_{k,p}({\fPsi}^{S,3}_{m})&=(2\pi-4\xi_2)^{-k_1-k_2},
\end{aligned}
\end{equation}
where $\xi_1,\xi_2$ are constants in Corollary~\ref{cor:CPhiPsiS}.
\end{corollary}

In two-dimensional case, Daubechies wavelets can be represented by
\[
\begin{array}{ll}
{\fPsi}^{D,1}_{m}(x_1,x_2):=\psi_m^D(x_1)\varphi_m^D(x_2),\\
{\fPsi}^{D,2}_{m}(x_1,x_2):=\varphi_m^D(x_1)\psi_m^D(x_2),\\
{\fPsi}^{D,3}_{m}(x_1,x_2):=\psi_m^D(x_1)\psi_m^D(x_2).
\end{array}
\]Similarly, , we have the
following result.
\begin{corollary}\label{cor:2DDaub}
Let ${\fPsi}^{D,1}_{m}$, ${\fPsi}^{D,2}_{m}$, and
${\fPsi}^{D,3}_{m}$ be defined in \eqref{def:2DDaub}. Then,
\begin{equation}
\begin{aligned}
\lim_{m\rightarrow\infty}C_{k,p}({\fPsi}^{D,1}_{m})&=\pi^{-k_1}\left(\frac{1-2^{1-pk_1}}{pk_1-1}\right)^{1/p}\frac{\pi^{-k_2}}{(1-pk_2)^{1/p}}, \,\,k_2\le 1/p, k_1\in\N,\\
\lim_{m\rightarrow\infty}C_{k,p}({\fPsi}^{D,2}_{m})&=\pi^{-k_2}\left(\frac{1-2^{1-pk_2}}{pk_2-1}\right)^{1/p}\frac{\pi^{-k_1}}{(1-pk_1)^{1/p}},\,\,k_1\le 1/p, k_2\in\N,\\
\lim_{m\rightarrow\infty}C_{k,p}({\fPsi}^{D,3}_{m})&=\pi^{-k_1-k_2}\left(\frac{1-2^{1-pk_1}}{pk_1-1}\cdot\frac{1-2^{1-pk_2}}{pk_2-1}\right)^{1/p},\,\,
(k_1,k_2)\in\N^2.\\
\end{aligned}
\end{equation}
\end{corollary}

\end{document}